\documentclass[10pt,fleqn]{article}
\usepackage[dvips]{epsfig}
\input{epsf}
\usepackage{amsmath,amssymb,amsfonts}
\usepackage{amscd}
\usepackage[dvips,matrix,arrow]{xy}

\numberwithin{equation}{subsection}
\numberwithin{figure}{subsection}

\newtheorem{theorem}{Theorem}[subsection]
\newtheorem{lemma}[theorem]{Lemma}
\newtheorem{remark}[theorem]{Remark}
\newtheorem{proposition}[theorem]{Proposition}
\newtheorem{corollary}[theorem]{Corollary}
\newtheorem{example}[theorem]{Example}

\def\var{\operatorname{Var}}%
\def\inf{\operatorname{inf}}%
\def\skew{\operatorname{Skew}}%
\def\kurt{\operatorname{Kurt}}%

\newenvironment{proof}[1][Proof:]{\begin{trivlist}
\item[\hskip \labelsep {\bfseries #1}]}{\end{trivlist}}

\newcommand{\qed}{\nobreak \ifvmode \relax \else
      \ifdim\lastskip<1.5em \hskip-\lastskip
      \hskip1.5em plus0em minus0.5em \fi \nobreak
      \vrule height0.75em width0.5em depth0.25em\fi}

 0 3 \font\ccc =msbm10 

\date{ }

\begin{document}

\title{On a Generalization of Markowitz \\Preference Relation}

\author{Valentin Vankov Iliev\footnote{Partially
supported by Grand I 02/18 of the Bulgarian Science Fund.}\\
Institute of Mathematics and Informatics,\\
Bulgarian Academy of Sciences, Sofia 1113, Bulgaria}

\maketitle

\begin{abstract}

Given two families of continuous functions $u=(u_p)_{p\in I}$ and
$v=(v_q)_{q\in J}$ on a topological space $X$, we define a preorder
$R=R(u,v)$ on $X$ by the condition that any member of $u$ is an
$R$-increasing and any member of $v$ is an $R$-decreasing function.
It turns out that if the topological space $X$ is quasi-compact and
sequentially compact, then any element $x\in X$ is $R$-dominated by
an $R$-maximal element $m\in X$: $xRm$. In particular, since the
$(n-1)$-dimensional simplex is a compact subset of $\hbox{\ccc
R}^n$, then considering its members as portfolios consisting of $n$
financial assets, we obtain the classical 1952 result of Harry
Markowitz that any portfolio is dominated by an efficient portfolio.
Moreover, several other examples of possible application of this
general setup are presented.

\end{abstract}

\section{Markowitz Optimization}

\label{5}

\subsection{Return of a Portfolio}

\label{1.1}

Let $\Delta_{n-1}=\{\left(x_1,\ldots,x_n\right)\in\hbox{\ccc
R}_+^n\mid \sum_{i=1}^nx_i=1\}$ be the $n-1$-dimensional simplex
and let $[n]=\{1,\ldots,n\}$. The ordered pairs $([n],x)$,
$x\in\Delta_{n-1}$, are sample spaces with set of outcomes $[n]$
and probability assignment $x\colon [n]\to\hbox{\ccc R}$,
$x(i)=x_i$, $i=1,\ldots,n$. The set of all sample spaces of this
form can be identified with the $n-1$-dimensional simplex
$\Delta_{n-1}$ and also are said to be \emph{$(n-1)$-dimensional
lotteries} or \emph{$(n-1)$-dimensional portfolios}.

Given a sample space $S$ with probability $P$, let $s_1,\ldots,s_n$
be random variables on $S$ with expected values
$\mu_1,\ldots,\mu_n$, respectively. For any portfolio
$x\in\Delta_{n-1}$ the weighted sum $s(x)=x_1s_1+\cdots +x_ns_n$ is
a random variable with expected value $u(x)=E(s(x))=x_1\mu_1+\cdots
+x_n\mu_n$ and the variance $v(x)=\var(s(x))$ is a non-negative
quadratic form in $x_1,\ldots,x_n$.

\begin{remark} \label{1.1.1} {\rm Below we interpret $i\in [n]$
as \emph{financial assets}, the sample space $S$ as a
\emph{financial market}, the random variables $s_i$ on $S$ as
\emph{returns on asset $i$}, $i=1,\ldots,n$, in the end of a fixed
time period, and $s(x)$ as the \emph{return of the portfolio $x$}.
Then $u(x)=E(s(x))$ is the \emph{expected return} and
$v(x)=\var(s(x))$ is the \emph{risk} (or, the \emph{volatility})
of the portfolio $x$ --- see, for example,~\cite[2.1]{[10]}.

}

\end{remark}

\subsection{Markowitz Preferences}

\label{5.1}

Let $x\in\Delta_{n-1}$ be a portfolio and $u(x)=E(s(x))$ and
$v(x)=\var(s(x))$ be the expected return and the volatility of
$x$. The Markowitz's approach to portfolio selection is based on
the following definition of preference $R$ on the set
$\Delta_{n-1}$ of portfolios: $xRy$ if $u(x)\leq u(y)$ and
$v(y)\leq v(x)$. Non-formally, $xRy$ means that the portfolio $y$
is at least as good as $x$. The symmetric part $E$ of the preorder
$R$ is
\[
E=\{(x,y)\in \Delta_{n-1}^2\mid u(x)=u(y)\hbox{\rm\ and\
}v(y)=v(x)\}
\]
and the asymmetric part $F$ of $R$ is $F=R\backslash E$. Thus,
$xFy$ if and only if either $u(x)<u(y)$ and $v(y)\leq v(x)$ or
$u(x)\leq u(y)$ and $v(y)<v(x)$. Non-formally, $xFy$ means  that
the portfolio $y$ is definitely better than the portfolio $x$.

In~\cite[p. 82]{[5]} H.~Markowitz gives (up to notation) the
following definition:

The portfolio $x$ is said to be \emph{efficient} if
\begin{equation}
u(x)=\max_{y\in\Delta_{n-1},v\left(y\right)\leq
v\left(x\right)}u(y)\hbox{\rm\ and\ }
v(x)=\min_{y\in\Delta_{n-1},u\left(y\right)\geq
u\left(x\right)}v(y).\label{5.1.1}
\end{equation}
In other words, for any portfolio $y\in\Delta_{n-1}$ the inequality
$v(y)\leq v(x)$ implies the inequality $u(x)\geq u(y)$ and the
inequality $u(y)\geq u(x)$ implies the inequality $v(x)\leq v(y)$.
The negation of the last statement is: There exists
$y\in\Delta_{n-1}$ such that $xFy$, that is, the portfolio $x$ is
not $R$-maximal.

Thus, we see that $x$ is Markowitz's efficient portfolio if and only
if $x$ is $R$-maximal --- this is our setup.

\section{Generalization}

\label{2}

In this section we present a wide generalization of Markowitz's
preference relation, defined in~\ref{5.1}. Using Kuratowski-Zorn
Theorem (equivalent to the Axiom of Choice), we show that any member
of this preference structure is dominated by a maximal element
(\emph{generalized efficient portfolio}). In particular, the set of
generalized efficient portfolios is not empty.

\subsection{A Preorder on a Topological Space}

\label{2.1}

Let $X$ be a topological space and let $u=(u_p)_{p\in I}$ and
$v=(v_q)_{q\in J}$ be two families of continuous real functions on
$X$. We define a preorder $R=R(u,v)$ on $X$ in the following way:
\begin{equation}
R=\{(x,y)\in X^2\mid u_p(x)\leq u_p(y)\hbox{\rm\ and\ } v_q(x)\geq
v_q(y)\hbox{\rm\ for all\ }p\in I,q\in J\}.\label{2.1.1}
\end{equation}
Then for the symmetric part $E$ of $R$ (an equivalence relation) one
has
\[
E=\{(x,y)\in X^2\mid u_p(x)=u_p(y)\hbox{\rm\ and\ } v_q(x)=
v_q(y)\hbox{\rm\ for all\ }p\in I,q\in J\}
\]
and for the asymmetric part $F$ of $R$ (an asymmetric and transitive
relation) one has $F=R\backslash E$. Thus, $xFy$ means $xRy$ and
either there exists index $p_0\in I$ with $u_{p_0}(x)<u_{p_0}(y)$ or
there exists index $q_0\in J$ with $v_{q_0}(x)>v_{q_0}(y)$.

On the account of repetitions of functions within one family and
adding the negatives of functions from one family to the other, we
can assume that both families have the same set of indices,
$u=(u_p)_{p\in I}$, $v=(v_p)_{p\in I}$, without changing the
corresponding preorder on $X$. Moreover, on the account of adding a
third countable family of continuous functions on $X$ to both
families, the corresponding preorder can be defined by two systems
of inequalities and a system of equalities.

Below, if the opposite is in not stated, the families $u=(u_p)_{p\in
I}$ and $v=(v_p)_{p\in I}$ have the same index set.

\subsection{Maximal Elements}

\label{2.5}

In order to fix the terminology, we remind several definitions. A
topological space $X$ is called \emph{quasi-compact} if every open
covering of $X$ contains a finite open covering. The space $X$ is
called \emph{compact} if it is quasi-compact and Hausdorff, and
\emph{sequentially compact} if any infinite sequence of elements
of $X$ has a converging subsequence.

It is well known (see, for example,~\cite[Sec. 1]{[15]}) that any
compact and first countable space is sequentially compact and that
every Lindel\"of, sequentially compact (and Hausdorf) space is
quasi-compact (compact).

Given a prerder $R$ on the set $X$, a subset $C\subset X$ is said to
be \emph{chain} in $X$ if the induced preorder on $C$ is complete. A
preordered set $X$ is called \emph{inductive} if every chain in $X$
has an upper bound.

Below, if the opposite is not stated, we suppose that the
topological space $X$ is furnished with the preorder $R$ produced by
the families of continuous functions $u=(u_p)_{p\in I}$ and
$v=(v_p)_{p\in I}$.

The sequence $(x_\iota)_{\iota=1}^\infty$, $x_\iota\in X$, is said
to be \emph{$R$-increasing} (respectively, \emph{strictly
$R$-increasing}) if $x_\iota Rx_{\iota +1}$ (respectively, $x_\iota
Fx_{\iota +1}$) for all $\iota\geq 1$. By analogy, we define
\emph{$R$-decreasing} (respectively, \emph{strictly $R$-decreasing})
sequences.

Given an $R$-chain $C\subset X$, for any $p\in I$ and any real
number $r\in\hbox{\ccc R}$ we set:
\[
M_p=\sup_{x\in C}u_p(x),\hbox{\ } m_p=\inf_{x\in C}v_p(x),
\]
\[
C_p=\{x\in C\mid u_p(x)=M_p\}, \hbox{\ }C_p^{\left
(-\right)}=\{x\in C\mid u_p(x)<M_p\},
\]
\[
c_p=\{x\in C\mid v_p(x)=m_p\},\hbox{\ }c_p^{\left
(+\right)}(r)=\{x\in C\mid v_p(x)>m_p\}.
\]
Finally, we denote $C_p^*=\{x\in X\mid u_p(x)=M_p\}$,
$c_p^*=\{x\in X \mid v_p(x)=m_p\}$, so $C_p\subset C_p^*$ and
$c_p\subset c_p^*$. Note that $C=C_p\cup C_p^{\left
(-\right)}=c_p\cup c_p^{\left (+\right)}$.

\begin{lemma}\label{2.5.5} Let $p,q\in I$.

{\rm (i)} One has $c_p\subset C_p$ or $C_p\subset c_p$.

{\rm (ii)} One has $c_p\cap C_p\subset c_q\cap C_q$ or $c_q\cap
C_q\subset c_p\cap C_p$.

\end{lemma}

\begin{proof} {\rm (i)} If $v_p(x)=m_p$ for all $x\in C_p$,
then $C_p\subset c_p$. Otherwise, there exists $x\in C_p$ with
$v_p(x)>m_p$ and, hence, $v_p(y)<v_p(x)$ for all $y\in c_p$. Since
any $y\in c_p$ is $R$-comparable with $x$, we have $u_p(y)\geq
u_p(x)=M_p$, that is, $y\in C_p$. In other words, $c_p\subset C_p$.

{\rm (ii)} If $v_q(x)=m_q$ and $u_q(x)=M_q$ for all $x\in c_p\cap
C_p$, then $c_p\cap C_p\subset c_q\cap C_q$. Otherwise, there
exists $x\in c_p\cap C_p$ with $v_q(x)>m_q$ or $u_q(x)<M_q$. If
$v_q(x)>m_q$ (respectively, $u_q(x)<M_q$), then $v_q(y)<v_q(x)$
(respectively, $u_q(x)<u_q(y)$) for all $y\in c_q\cap C_q$. Since
$x$ and $y$ are $R$-comparable, in both cases we have $u_p(y)\geq
u_p(x)=M_p$ and $m_p=v_p(x)\geq v_p(y)$. In other words, $y\in
c_p\cap C_p$ for all $y\in c_q\cap C_q$.

\end{proof}

Let us fix a positive integer $s$ and a finite subset
$\{p_1,\ldots,p_s\}\subset I$.

Using Lemma~\ref{2.5.5}, {\rm (i)}, {\rm (ii)}, and induction, we
obtain immediately the following:

\begin{corollary}\label{2.5.10} The intersection $c_{p_1}\cap
C_{p_1}\cap\ldots\cap c_{p_k}\cap C_{p_k}$ is equal to one of the
sets $c_{p_1}$, $C_{p_1}$, $\ldots$, $c_{p_k}$, $C_{p_k}$ for all
$k\leq s$.

\end{corollary}

Given an $s\geq 1$, in accord with Lemma~\ref{2.5.5}, {\rm (i)},
{\rm (ii)}, and eventual renumbering of the pairs of functions
$u_{p_k}, v_{p_k}$, we order the intersections $c_{p_k}\cap
C_{p_k}$, $k\leq s$, with respect to inclusion from smallest to
largest:
\begin{equation}
c_{p_1}\cap C_{p_1}\subset\cdots\subset c_{p_\ell}\cap
C_{p_\ell}\subset c_{p_{\ell+1}}\cap
C_{p_{\ell+1}}\subset\cdots\subset c_{p_s}\cap
C_{p_s},\label{2.5.15}
\end{equation}
where $c_{p_i}=\emptyset$ or $C_{p_i}=\emptyset$, $1\leq
i\leq\ell$, and $c_{p_{\ell+1}}\cap C_{p_{\ell+1}}\neq\emptyset$.
Below, if the opposite is not stated, after fixing
$\{p_1,\ldots,p_s\}\subset I$, we assume that~(\ref{2.5.15})
holds.

Thus, the existence of $k\leq s$ with $c_{p_k}=\emptyset$ or
$C_{p_k}=\emptyset$ after renumbering implies $\ell\geq 1$, that is,
$c_{p_1}=\emptyset$ or $C_{p_1}=\emptyset$.

\begin{lemma}\label{2.5.20} Let $X$ be a sequentially compact
space and let $C_{p_1}=\emptyset$ (respectively,
$c_{p_1}=\emptyset$).

{\rm (i)} There exists a strictly $R$-increasing and divergent
sequence
\begin{equation}
(x_\iota)_{\iota=1}^\infty,\label{2.5.25}
\end{equation}
with $x_\iota\in C$ and limit $x^*\in X$, such that the sequence of
real numbers $(u_{p_1}(x_\iota))_{\iota=1}^\infty$ is strictly
increasing and diverges to $u_{p_1}(x^*)=M_{p_1}$ and every sequence
of real numbers $(v_q(x_\iota))_{\iota=1}^\infty$, $q\in I$, is
decreasing and diverges to $v_q(x^*)=m_q$ (respectively, the
sequence of real numbers $(v_{p_1}(x_\iota))_{\iota=1}^\infty$ is
strictly decreasing and diverges to $v_{p_1}(x^*)=m_{p_1}$ and every
sequence of real numbers $(u_q(x_\iota))_{\iota=1}^\infty$, $q\in
I$, is increasing and diverges to $u_q(x^*)=M_q$).

{\rm (ii)} Let for the sequence~(\ref{2.5.25}) from part {\rm (i)}
one has $u_{p_1}(x^*)=M_{p_1}$, $u_{p_2}(x^*)=M_{p_2}$,$\ldots,$
$u_{p_k}(x^*)=M_{p_k}$ (respectively, $v_{p_1}(x^*)=m_{p_1}$,
$v_{p_2}(x^*)=m_{p_2}$,$\ldots,$ $v_{p_k}(x^*)=m_{p_k}$), for some
$k<s$. Then either there exists $y\in \cap_{\lambda\in I}
c_\lambda\cap C_{p_1}\cap\ldots\cap C_{p_k}\cap C_{p_{k+1}}$
(respectively, $y\in c_{p_1}\cap\ldots\cap c_{p_k}\cap
c_{p_{k+1}}\cap_{\lambda\in I}C_\lambda$), or there exists a
strictly $R$-increasing and divergent sequence
$(y_\kappa)_{\kappa=1}^\infty$, with $y_\kappa\in C$ and limit
$y^*\in X$, such that $u_{p_1}(y^*)=M_{p_1}$,
$u_{p_2}(y^*)=M_{p_2}$,$\ldots,$ $u_{p_k}(y^*)=M_{p_k}$, and
$v_q(y^*)=m_q$, $q\in I$ (respectively, $v_{p_1}(y^*)=m_{p_1}$,
$v_{p_2}(y^*)=m_{p_2}$,$\ldots,$ $v_{p_k}(y^*)=m_{p_k}$, and
$u_q(x^*)=M_q$, $q\in I$), the sequence of real numbers
$(u_{p_{k+1}}(y_\kappa))_{\kappa=1}^\infty$ is strictly increasing
and diverges to $u_{p_{k+1}}(y^*)=M_{p_{k+1}}$ and every sequence of
real numbers $(v_q(y_\kappa))_{\kappa=1}^\infty$, $q\in I$, is
decreasing and diverges to $v_q(y^*)=m_q$ (respectively, the
sequence of real numbers $(v_{p_{k+1}}(y_\kappa))_{\kappa=1}^\infty$
is strictly decreasing and diverges to
$v_{p_{k+1}}(y^*)=m_{p_{k+1}}$ and every sequence of real numbers
$(u_q(y_\kappa))_{\kappa=1}^\infty$, $q\in I$, is increasing and
diverges to $u_q(y^*)=M_q$).

\end{lemma}

\begin{proof} Below, when $c_{p_1}=\emptyset$, we replace
$u_q$ with $-v_q$, $v_q$ with $-u_q$, and use the corresponding
proofs in case $C_{p_1}=\emptyset$.

{\rm (i)} Let $C_{p_1}=\emptyset$. Then $M_{p_1}=\sup_{x\in
C_{p_1}^{\left (-\right)}}u_{p_1}(x)$ and we choose
$(x_\iota)_{\iota=1}^\infty$ to be a sequence of members of
$C=C_{p_1}^{\left (-\right)}$ such that the sequence of real numbers
$(u_{p_1}(x_\iota))_{\iota=1}^\infty$ is strictly increasing with
$\lim_{\iota\to\infty}u_{p_1}(x_\iota)=M_{p_1}$. Since the elements
$x_\iota$ $\iota\geq 1$, are pairwise $R$-comparable, it turns out
that the sequences of real numbers
$(u_q(x_\iota))_{\iota=1}^\infty$, $q\in I$, $q\neq p_1$, are
increasing and $(v_q(x_\iota))_{\iota=1}^\infty$, $q\in I$, are
decreasing. Thus, the sequence $(x_\iota)_{\iota=1}^\infty$ is
strictly $R$-increasing. In accord with the sequential compactness
of the topological space $X$, we can suppose that
$(x_\iota)_{\iota=1}^\infty$ diverges to a point $x^*\in X $. Thus,
$u_{p_1}(x^*)=M_{p_1}$. For any $q\in I$ we set
$m'_q=\lim_{\iota\to\infty}v_q(x_\iota)$. Let us suppose
$m_{q_0}<m'_{q_0}$ for some $q_0\in I$ and let $y\in C$ be such that
$v_{q_0}(y)<m'_{q_0}$. In particular, $v_{q_0}(y)<v_{q_0}(x_\iota)$,
hence $u_{p_1}(y)\geq u_{p_1}(x_\iota)$ for all $\iota\geq 1$.
Taking the limit we obtain $u_{p_1}(y)\geq M_{p_1}$, that is, $y\in
C_{p_1}$, which is a contradiction. Therefore $m_q=m'_q$ and
$v_q(x^*)=m_q$ for all $q\in I$.

{\rm (ii)} Let
$M'_{p_{k+1}}=\lim_{\iota\to\infty}u_{p_{k+1}}(x_\iota)$. We have
$M'_{p_{k+1}}\leq M_{p_{k+1}}$ and if $M'_{p_{k+1}}=M_{p_{k+1}}$,
then $u_{p_{k+1}}(x^*)=M_{p_{k+1}}$. In other words,
$x^*\in\cap_{\lambda=1}^\infty c_\lambda^*\cap
C_{p_1}^*\cap\ldots\cap C_{p_k}^*\cap C_{p_{k+1}}^*$. Now, let
$M'_{p_{k+1}}<M_{p_{k+1}}$.

In case $C_{p_{k+1}}\neq\emptyset$, we choose $y\in C_{p_{k+1}}$ and
since $x_\iota$'s and $y$ are $R$-comparable, the inequalities
$u_{p_{k+1}}(x_\iota)\leq M'_{p_{k+1}}<u_{p_{k+1}}(y)$ yield
\begin{equation}
u_q(x_\iota)\leq u_q(y)\label{2.5.30}
\end{equation}
for all $q\in I$, $q\neq p_{k+1}$, and
\begin{equation}
v_q(x_\iota)\geq v_q(y)\label{2.5.35}
\end{equation}
for all $q\in I$. Taking the limit $\iota\to\infty$
in~(\ref{2.5.30}) for all $q=p_1,\ldots,p_k$ and
in~(\ref{2.5.35}) for all $q\in I$, we obtain
$y\in\cap_{\lambda=1}^\infty c_\lambda\cap C_{p_1}\cap\ldots\cap
C_{p_k}\cap C_{p_{k+1}}$.

In case $C_{p_{k+1}}=\emptyset$, there exists a sequence
$(y_\kappa)_{\kappa=1}^\infty$, $y_\kappa\in C$, such that
$M'_{p_{k+1}}<u_{p_{k+1}}(y_\kappa)<M_{p_{k+1}}$, $\kappa\geq 1$,
the sequence of real numbers
$(u_{p_{k+1}}(y_\kappa))_{\kappa=1}^\infty$ is strictly increasing
and diverges to $M_{p_{k+1}}$. In particular,
$u_{p_{k+1}}(x_\iota)<u_{p_{k+1}}(y_\kappa)$ for all
$\iota,\kappa\geq 1$. Since $x_\iota$'s and $y_\kappa$'s are
$R$-comparable, we obtain for all $\iota,\kappa\geq 1$ the
inequalities
\begin{equation}
u_q(x_\iota)\leq u_q(y_\kappa)\leq M_q\label{2.5.40}
\end{equation}
for all $q\neq p_{k+1}$, and
\begin{equation}
v_q(x_\iota)\geq v_q(y_\kappa)\geq m_q\label{2.5.45}
\end{equation}
for all $q\in I$. Since the topological space $X$ is sequentially
compact, we can assume that $(y_\kappa)_{\kappa=1}^\infty$
diverges with limit $y^*\in X$, so $u_{p_{k+1}}(y^*)=M_{p_{k+1}}$.
Taking consecutively the limits $\iota\to\infty$,
$\kappa\to\infty$, in~(\ref{2.5.40}) for all $q=p_1,\ldots,p_k$
and in~(\ref{2.5.45}) for all $q\in I$, we obtain
$y^*\in\cap_{\lambda=1}^\infty c_\lambda^*\cap
C_{p_1}^*\cap\ldots\cap C_{p_k}^*\cap C_{p_{k+1}}^*$.

\end{proof}

\begin{proposition}\label{2.5.50} Let $X$ be a sequentially compact
space endowed with the preorder $R$ from~(\ref{2.1.1}) and let
$C\subset X$ be a chain.

{\rm (i)} For any finite subset $\{p_1,\ldots,p_s\}\subset I$ one
has
\begin{equation}
\cap_{i=1}^sC_{p_i}^*\cap c_{p_i}^*\neq\emptyset.\label{2.5.55}
\end{equation}

{\rm (ii)} If $X$ is, in addition, quasi-compact, then
\begin{equation}
\cap_{p\in I}C_p^*\cap c_p^*\neq\emptyset.\label{2.5.60}
\end{equation}

\end{proposition}

\begin{proof}  {\rm (i)} If $C$ is a finite $R$-chain, then
its largest element is a member of the intersection
$\cap_{i=1}^sC_i\cap c_i$.

Now, let us suppose that the $R$-chain $C$ is infinite. In case all
sets $c_1$, $C_1$, $\ldots$, $c_s$, $C_s$, are nonempty
Corollary~\ref{2.5.10} implies that their intersection is not empty,
hence~(\ref{2.5.60}) holds. Otherwise, using Lemma~\ref{2.5.20} and
induction with respect to $k$, we are done.

{\rm (ii)} Since $X$ is quasi-compact, part {\rm (i)} implies part
{\rm (ii)}.

\end{proof}

\begin{corollary}\label{2.5.65} If $X$ is a quasi-compact and
sequentially compact space, then the preordered set $X$ is
inductive.

\end{corollary}

\begin{proof} Every element $x^*\in\cap_{p\in I} C_p^*\cap
c_p^*$ is an upper bound of the $R$-chain $C$, hence the
preordered set $X$ is inductive.

\end{proof}

Now, Corollary~\ref{2.5.65} and Kuratowski-Zorn Theorem yield the
following:

\begin{theorem} \label{2.5.70} Let $X$ be a quasi-compact and
sequentially compact space. For any element $x\in X$ there exists an
$R$-maximal element $y\in X$ with $xRy$.

\end{theorem}

\subsection{Examples}

\label{1.30}

Since the $(n-1)$-dimensional simplex $\Delta_{n-1}$ is a compact
set in $\hbox{\ccc R}^n$, it is a quasi-compact and sequentially
compact topological space. In case the family $u$ consists of one
function $u(x)$ --- the expected return of the portfolio $x$ and
the family $v$ consists of one function $v(x)$ --- its volatility,
using Theorem~\ref{2.5.70}, we obtain the existence of Markowitz
efficient portfolios and something more: Any portfolio is
$R$-dominated by a Markowitz efficient portfolio.

Moreover, replacing the simplex $\Delta_{n-1}$ with a closed ball
$B_{n-1}$ in the affine hyperplane $\sum_{i=1}^nx_i=1$ in
$\hbox{\ccc R}^n$, such that $\Delta_{n-1}\subset B_{n-1}$, we admit
bounded negative $x_i$'s (that is, constrained \emph{short sales})
and again Theorem~\ref{2.5.70} assures existence of Markowitz
efficient portfolios which dominate any given portfolio.

Below, we remind some notions from statistics and give examples of
application of Theorem~\ref{2.5.70}.

Given the integer $\ell\geq 2$, the \emph{$\ell$-th central moment}
of the random variable $s(x)$ is $E((s(x)-E(s(x)))^\ell)$. The
standard variance is the second central moment
$v(x)=E((s(x)-E(s(x)))^2)$ of $s(x)$ and it is a quadratic form in
$x_1,\ldots, x_n$. The third central moment $E((s(x)-E(s(x)))^3)$ is
a cubic form and the fourth central moment $E((s(x)-E(s(x)))^4)$ is
a form of degree $4$ in $x_1,\ldots, x_n$.

Given $x\in\Delta_{n-1}$ and $t\in\hbox{\ccc R}$, we set
$F_x(t)=P(\{m\in S\mid s(x)(m)<t\})$, so $F_x\colon\hbox{\ccc R}\to
[0,1]$ is the \emph{cumulative distribution function} of the random
variable $s(x)$. We assume that $s(x)$ is a continuous random
variable with \emph{density function} $f_x(t)$, so
$F_x(t)=\int_{-\infty}^tf_x(\tau)d\!\tau$ and $F'_x(t)=f_x(t)$. In
particular, the functions $F_x(t)$ are continuous.

We define recursively $D_x^{\left(1\right)}(t)=F_x(t)$,
$D_x^{\left(2\right)}(t)=\int_{-\infty}^tF_x(\tau)d\!\tau$,$\ldots,$
$D_x^{\left(\ell\right)}(t)=\int_{-\infty}^tD_x^{\left(\ell-1\right)}
(\tau)d\!\tau$, $\ldots$.

The portfolio $x\in\Delta_{n-1}$ is said to be \emph{$\ell$-th
order stochastically dominated} by portfolio $y\in\Delta_{n-1}$ if
$D_y^{\left(\ell\right)}(t)\leq D_x^{\left(\ell\right)}(t)$ for
all $t\in\hbox{\ccc R}$. In case the previous inequalities hold
and $D_y^{\left(\ell\right)}(t) < D_x^{\left(\ell\right)}(t)$ for
some $t\in\hbox{\ccc R}$, $x$ is said to be \emph{$\ell$-th order
strictly stochastically dominated} by $y$.

We set
\[
\skew(s(x))=\frac{E((s(x)-E(s(x)))^3)}{\var(s(x))^{\frac{3}{2}}}
\]
 to be the \emph{skewness} and
\[
\kurt(s(x))=\frac{E((s(x)-E(s(x)))^4)}{\var(s(x))^2}-3
\]
to be the \emph{kurtosis}, or, \emph{excess kurtosis} of the random
variable $s(x)$.

If the random variable $s(x)$ is normal, then $\skew(s(x))=
\kurt(s(x))=0$.

\begin{example}\label{1.30.1} {\rm In case $I=\{1\}$, $J=\emptyset$,
the function $u=u_1$ can be considered as an utility function on
$\Delta_{n-1}$ and $R$ is the corresponding preference relation with
negatively transitive asymmetric part $F$.

}

\end{example}

\begin{example}\label{1.30.10} {\rm In case $I=\{1\}$, $J=\{1\}$,
\[
u_1(x)=E(s(x)),
\]
\[
v_1(x)=\var(s(x)),
\]
we obtain the classical Markowitz setup.

}

\end{example}

\begin{example}\label{1.30.20} {\rm In case
\[
u_1(x)=E(s(x)),
\]
\[
v_1(x)=\var(s(x)),\hbox{\ } v_2(x)=\skew^2(s(x)),
\]
we simultaneously maximize the expected return $E(s(x))$ and
minimize the volatility $\var(s(x))$ and the absolute value of the
skewness $\skew(s(x))$ of the return $s(x)$ of the portfolio $x$.

}

\end{example}

\begin{example}\label{1.30.25} {\rm In case
\[
u_1(x)=E(s(x)),
\]
\[
v_1(x)=\var(s(x)),\hbox{\ } v_2(x)=\kurt^2(s(x)),
\]
we simultaneously maximize the expected return $E(s(x))$ and
minimize the volatility $\var(s(x))$ and the absolute value of the
kurtosis $\kurt(s(x))$ of the return $s(x)$, thus balancing the
tails of its distribution.

}

\end{example}

\begin{example}\label{1.30.30} {\rm In case
\[
u_1(x)=E(s(x)),
\]
\[
v_1(x)=\var(s(x)),\hbox{\ } v_2(x)=\skew^2(s(x)),\hbox{\
}v_3(x)=\kurt^2(s(x)),
\]
we simultaneously maximize the expected return $E(s(x))$ and
minimize the volatility $\var(s(x))$, the the absolute value of the
skewness $\skew(s(x))$, and the absolute value of the kurtosis
$\kurt(s(x))$ of the return $s(x)$. In this way we balance both the
tails of the distribution of $s(x)$ and "round" the maximum of its
density function $f_x(t)$.

}

\end{example}

\begin{example}\label{1.30.40} {\rm  In case
\[
v_t(x)=D_x^{\left(\ell\right)}(t),\hbox{\ } t\in\hbox{\ccc R},
\]
we maximize the $\ell$-th order stochastic dominance, $\ell\geq 1$.

}

\end{example}

\begin{example}\label{1.30.45} {\rm  In case
\[
u(x)=E(s(x)),
\]
\[
v(x)=\var(s(x)),\hbox{\ }
v_t(x)=D_x^{\left(\ell\right)}(t),\hbox{\ } t\in\hbox{\ccc R},
\]
we simultaneously maximize the expected return $E(u(x))$ and the
$\ell$-th order stochastic dominance, $\ell\geq 1$, and minimize the
volatility $\var(s(x))$.

}

\end{example}

\begin{example}\label{1.30.50} {\rm  Let $X$ be a quasi-compact and
sequentially compact space and let $f\colon X\times X\to \hbox{\ccc
R}$ be a continuous real function. For any $p\in X$ we set
\[
u_p(x)=f(x,p),\hbox{\ }x\in X,
\]
\[
v_p(y)=f(p,y),\hbox{\ }y\in X.
\]
Further, for any $x\in X$ we set
\[
U_x^{\left (\geq\right)}=\{y\in X\mid f\left(y,p\right)\geq
f\left(x,p\right)\hbox{\rm\ for all\ }p\in X\},
\]
\[
V_x^{\left (\leq\right)}=\{y\in X\mid f\left(p,y\right)\leq
f\left(p,x\right)\hbox{\rm\ for all\ }p\in X\},
\]
and for any $x,p\in X$ we set
\[
U_x^{\left (\hat{p};\geq\right)}=\{y\in X\mid f\left(y,q\right)\geq
f\left(x,q\right)\hbox{\rm\ for all\ }q\in X,q\neq p\},
\]
\[
V_x^{\left (\hat{p};\leq\right)}=\{y\in X\mid f\left(q,y\right)\leq
f\left(q,x\right)\hbox{\rm\ for all\ }q\in X,q\neq p\}.
\]
Note that $U_x^{\left (\geq\right)}$, $V_x^{\left (\leq\right)}$,
$U_x^{\left (\hat{p};\geq\right)}$, $V_x^{\left
(\hat{p};\leq\right)}$, are closed subsets of $X$ and that
\[
x\in U_x^{\left (\geq\right)}\subset U_x^{\left
(\hat{p};\geq\right)},\hbox{\ } x\in V_x^{\left
(\leq\right)}\subset V_x^{\left (\hat{p};\leq\right)}
\]
for all $x,p\in X$. According to Theorem~\ref{2.5.70}, there
exists an element $m\in X$, such that for any $p\in X$ one has
\[
f(m,p)=\max_{y\in U_m^{\left (\hat{p};\geq\right)}\cap V_m^{\left
(\leq\right)}}f(y,p)
\]
and
\[
f(p,m)=\min_{y\in U_m^{\left (\geq\right)}\cap V_m^{\left
(\hat{p};\leq\right)}}f(p,y).
\]

}

\end{example}

\section*{Acknowledgement}

\label{3}

The friendship of Prof. Fernando Fernandez Rodrigues, Department of
Quantitative Methods on Economics and Management, Las Palmas de Gran
Canaria University, is much appreciated and has led to many
interesting and good-spirited discussions relating to this research.

\end{document}